\documentclass{amsart}      
\usepackage{amssymb,amsthm, amsmath, amsfonts} 
\usepackage{graphics}                 
\usepackage{hyperref}                 
\usepackage[all]{xy}
\usepackage{enumerate}

\newtheorem{theorem}{Theorem}[section]
\newtheorem{lemma}[theorem]{Lemma}
\newtheorem{proposition}[theorem]{Proposition}

\newtheorem{corollary}[theorem]{Corollary}

\newtheorem{example}[theorem]{Example}

\numberwithin{equation}{section}

\DeclareMathOperator{\rad}{rad}
\DeclareMathOperator{\Ob}{Ob}
\DeclareMathOperator{\pd}{pd}
\DeclareMathOperator{\id}{id}

\DeclareMathOperator{\Hom}{Hom}
\DeclareMathOperator{\End}{End}
\DeclareMathOperator{\Ext}{Ext}

\DeclareMathOperator{\supp}{supp}
\DeclareMathOperator{\Top}{Top}

\title[Stratifications of directed categories and APR tilting modules]{Stratifications of finite directed categories and generalized APR tilting modules}
\author{Liping Li}
\thanks{The author would like to thank the referee for carefully reading the preprint, and pointing out some existed results on this topic, which are unknown to the author.}
\address{Department of Mathematics, University of California, Riverside, CA, 92507.}
\email{lipingli@math.ucr.edu}
\begin{document}

\begin{abstract}
A finite directed category is a $k$-linear category with finitely many objects and an underlying poset structure, where $k$ is an algebraically closed field. This concept unifies structures such as $k$-linerizations of posets and finite EI categories, quotient algebras of finite-dimensional hereditary algebras, triangular matrix algebras, etc. In this paper we study representations of finite directed categories and discuss their stratification properties. In particular, we show the existence of generalized APR tilting modules for triangular matrix algebras under some assumptions.
\end{abstract}

\maketitle

\section{Introduction}

It is worth to point out that in representation theory many structures people are interested in have underlying posets. Specific examples include posets, directed quivers, quotient algebras of finite-dimensional hereditary algebras (in particular, piecewise hereditary algebras, see \cite{Happel,HRS}), Auslander algebras of representation-directed algebras, triangular matrix algebras (see \cite{BFT}), transporter categories (see \cite{Xu2}), orbit categories (\cite{Webb2}), fusion systems (\cite{Linckelmann}), and skeletal finite EI categories (i.e., finite categories such that every endomorphism is an isomorphism, see \cite{Dieck,Dietrich,Li1,Li2,Luck,Webb1,Webb2,Xu1,Xu2}). Therefore, it makes sense to define a concept unifying these structures, study their representations and homological properties, and generalize many existed but sporadic results.

This concept has been defined in \cite{Li3,Li4}, which we call \textit{finite directed categories}. By definition, a finite directed category $\mathcal{A}$ is a $k$-linear category with finitely many objects, where $k$ is an algebraically closed field, satisfying the following properties: $\mathcal{A}$ is \textit{locally finite}, i.e., for two objects $x, y \in \Ob \mathcal{A}$, $\mathcal{A} (x, y)$ is a finite-dimensional vector space; there is a partial order $\leqslant$ on $\Ob \mathcal{A}$ such that $\mathcal{A} (x, y) \neq 0$ implies $x \leqslant y$. Note that we can extend this partial order to a linear order with respect to which $\mathcal{A}$ is still directed. Indeed, let $O_1$ be the set of all minimal objects in $\Ob \mathcal{A}$; let $O_2$ be the set of all minimal objects in $\Ob \mathcal{A} \setminus O_1$, and so on. Define an arbitrary linear order $\leqslant_i$ for each set $O_i$. For two objects $x, y \in \Ob \mathcal{A}$, we then define $x < y$ if $x <_i y$ for some $i$, or $x \in O_i$, $y \in O_j$, and $i < j$. The order defined in this way is indeed linear, and $\mathcal{A}$ is directed with respect to it. Therefore, without loss of generality we assume that the partial order $\leqslant$ is linear. We also suppose that $\mathcal{A}$ is \textit{connected}. That is, for $x, y \in \Ob \mathcal{A}$, there is a sequence of objects $x = x_0, x_1, \ldots, x_n = y$ such that either $\mathcal{A} (x_i, x_{i+1}) \neq 0$ or $\mathcal{A} (x_{i+1}, x_i) \neq 0$, $0 \leqslant i \leqslant n-1$.

A \textit{representation} $R$ of $\mathcal{A}$ is a $k$-linear covariant functor from $\mathcal{A}$ to $k$-vec, the category of finite-dimensional vector spaces. Note that by Gabriel's construction (\cite{BG}), $\mathcal{A}$ (more precisely, the space of all morphisms in $\mathcal{A}$) can be viewed as a finite-dimensional algebra $A$, and the category of representations of $\mathcal{A}$ can be identified with $A$-mod, the category of finitely generated $A$-modules. \footnote{This result is true for all locally finite $k$-linear categories with finitely many objects, even if it is not directed.} We call $A$ the \textit{associated algebra} of $\mathcal{A}$, and call $\mathcal{A}$ the \textit{associated category} of $A$. By abuse of notation, we identify the category $\mathcal{A}$ with the algebra $A$, and call $R$ an $\mathcal{A}$-module.

It is clear from this definition that $k$-linearizations of finite posets, transporter categories, fusion systems, orbit categories, and skeletal finite EI categories are indeed directed categories. Furthermore, finite-dimensional hereditary algebras and their quotient algebras, and triangular matrix algebras can be viewed as directed categories in an obvious way. It is also clear from the definition that every directed category $\mathcal{A}$ is skeletal. However, the corresponding algebra $A$ might not be basic since for $x \in \Ob \mathcal{A}$, the endomorphism algebra $\mathcal{A} (x,x)$ might not be basic. In the case that $\mathcal{A} (x, x)$ is a local algebra, we call $x$ a \textit{primitive} object. If every object in $\mathcal{A}$ is primitive, then the associated algebra $A$ is basic.

In the next section we introduce some elementary results on representations of of directed categories, describe the indecomposable projective modules and simple modules, and study the induction and restriction functors with respect to full subcategories (which are also directed). Corresponding results for finite EI categories have been explored in \cite{Webb1,Xu1}.

Directed categories have nice stratification properties. Explicitly, every directed category is stratified with respect to a preorder $\preccurlyeq$ determined by the given linear order $\leqslant$ on $\Ob \mathcal{A}$, and standard modules with respect to $\preccurlyeq$ coincide with indecomposable summands of endomorphism algebras of objects. Directed categories standardly stratified with respect to $\preccurlyeq$ have been characterized in \cite{Li3}. In Section 3 we give more properties. In particular, we prove that the associated category of an arbitrary finite-dimensional algebra is a directed category with respect to a linear order if and only if the composition factors of every standard module with respect to this linear order are all isomorphic, if and only if all proper standard modules are simple. We also show that when every object in $\mathcal{A}$ is primitive, and $\mathcal{A}$ is standardly stratified with respect to $\leqslant$, then an $\mathcal{A}$-module $M$ has finite projective dimension if and only if it has a filtration by standard modules, if and only if its value $M(x)$ on each object $x \in \Ob \mathcal{A}$ is a free $\mathcal{A} (x, x)$-module. In other words, under the assumptions $\mathcal{F} (\Delta)$, the category of all finitely generated $\mathcal{A}$-modules with filtrations by standard modules, coincide with $\mathcal{P} ^f (\mathcal{A})$, the category of finitely generated $\mathcal{A}$-modules with finite projective dimension. The problem whether these two important subcategories of $\mathcal{A}$-mod coincide has been considered by Platzeck and Reiten in \cite{PR}.

Let $\Lambda$ be a finite-dimensional basic $k$-algebra and suppose that it has a simple projective module $S$. The APR tilting module is defined in \cite{APR} as $Q \oplus \tau ^{-1} S$, where $\tau$ is the Auslander-Reiten translation, and $Q$ is the direct sum of all indecomposable projective $A$-modules (up to isomorphism) except $S$. An observation tells us that under the given assumption $\Lambda \cong \begin{bmatrix} \Lambda_1 & 0 \\ M & k \end{bmatrix}$ is a triangular matrix algebra, and hence can be viewed as a directed category. It is natural to ask whether general APR tilting modules exist for arbitrary triangular matrix algebras $\begin{bmatrix} \Lambda_1 & 0 \\ M & \Lambda_2 \end{bmatrix}$ where $\Lambda_2$ is a local algebra. In the last section we show the existence of such APR tilting modules under suitable conditions.

We introduce the notation and convention here. Throughout this paper $\mathcal{A}$ is a connected directed category with respect to a fixed linear order $\leqslant$ on $\Ob \mathcal{A}$, and its associated algebra is denoted by $A$. Sometimes we consider an arbitrary algebra and denote it by $\Lambda$ to distinguish it from $A$. For every $x \in \Ob \mathcal{A}$, we let $1_x$ be the identity morphism, which is also an idempotent in $A$. The symbol $[n]$ is the set of all positive integers from 1 to $n$. All modules we consider in this paper are left finitely generated modules if we make no other claim. Composite of maps, morphisms and actions is from right to left. To simplify the expression of statements, we view the zero module as a projective or a free module.

\section{Preliminaries}

We first give some examples of directed categories. Let $\Lambda$ be a quotient algebra of a finite-dimensional hereditary algebra, and let $Q$ be the ordinary quiver. Then $\Lambda$ can be regarded as a finite directed category. Objects are just the vertices of $Q$, and morphisms from vertex $v$ to vertex $w$ are elements in $1_w \Lambda 1_v$.

By definition, a \textit{finite EI category} $\mathcal{E}$ is a small category with finitely many morphisms such that every endomorphism is an isomorphisms. Examples of finite EI categories includes finite posets, transporter categories \cite{Xu2}, orbit categories \cite{Webb2}, and fusion systems \cite{Linckelmann}. When $\mathcal{E}$ is skeletal, we can define a partial order $\preccurlyeq$ on $\Ob \mathcal{E}$ as follows: for $x, y \in \Ob \mathcal{E}$ such that $\mathcal{E} (x, y) \neq \emptyset$, we let $x \preccurlyeq y$. As we did in the introduction, we can extend this partial order to a linear order $\leqslant$, with respect to which the $k$-linearization of $\mathcal{E}$ is a directed category.

Let $A_1$ and $A_2$ be two finite-dimensional $k$-algebras and let $M$ be a $(A_2, A_1)$-bimodule. Then we can construct the triangular matrix algebra $A = \begin{bmatrix} A_1 & 0 \\ M & A_2 \end{bmatrix}$. The elements of $A$ are $2 \times 2$ matrices $\begin{bmatrix} a & 0 \\ v & b \end{bmatrix}$, where $a \in A_1, b \in A_2, v \in M$. Addition and multiplication are defined by the usual operations on matrices. For details, see \cite{ARS}. The associated category of $A$ is a directed category with the following structure:
\begin{equation*}
\xymatrix{ \mathcal{A}: & x \ar@(ld,lu)[]|{A_1} \ar @<1ex>[rr] ^M \ar@<-1ex>[rr] ^{\ldots} & & y \ar@(rd,ru)[]|{A_2}}.
\end{equation*}

Conversely, given a directed category $\mathcal{A}$, its associated algebra $A$ is a triangular matrix algebra. Indeed, let $x$ be a maximal object in $\mathcal{A}$ with respect to $\leqslant$, and let $\epsilon = \sum _{x \neq z \in \Ob \mathcal{A}} 1_z$. Define $A_1 = \epsilon A \epsilon$, $A_2 = 1_x A 1_x$, and $M = 1_x A \epsilon$. Note that $\epsilon A 1_x = \epsilon \mathcal{A} 1_x = 0$ since there is no nonzero morphisms from objects different from $x$ to $x$. Consequently, $A = \begin{bmatrix} A_1 & 0 \\ M & A_2 \end{bmatrix}$ is a triangular matrix algebra.

Let $\Lambda$ be a finite-dimensional algebra standardly stratified with respect to a linear order on isomorphism classes of simple modules, and let $\Gamma$ be the extension algebra of standard modules. In \cite{Li4} we show that the associated $k$-linear category of $\Gamma$ is a directed category with respect to this linear order. In \cite{Li5} we show that if $\Lambda$ is standardly stratified with respect to all linear orders on isomorphism classes of simple modules, then the associated category of $\Lambda$ is directed.

Let $\Lambda$ be a finite-dimensional algebra. A \textit{path} in $\Lambda$-mod is a sequence
\begin{equation*}
\xymatrix{M_0 \ar[r] ^{f_1} & M_1 \ar[r] ^{f_2} \ar[r] & \ldots \ar[r] ^{f_t} & M_t}
\end{equation*}
of nonzero nonisomorphisms $f_1, \ldots, f_t$, where all modules in this sequence are indecomposable. It is a \textit{cycle} if $M_0 \cong M_t$. A $\Lambda$-module is called a \textit{directed} module if it appears in no cycles. The algebra $\Lambda$ is called \textit{representation-directed} if every indecomposable $\Lambda$-module is directed. It is known that every representation-directed algebra has finite representation type. Conversely, if $\Lambda$ is a hereditary or tilted algebra of finite representation type, then it is representation-directed (see Lemma 1.1 and Corollary 3.4 in Chapter IX, \cite{ASS}.

The following proposition gives us a good relation between representation-directed algebras and directed categories. Recall that for an algebra $\Lambda$ of finite representation type, its \textit{Auslander algebra} is the endomorphism algebra of the direct sum of all indecomposable $\Lambda$-modules (up to isomorphism).

\begin{proposition}
Let $\Lambda$ be a finite-dimensional algebra of finite representation type and let $A$ be its Auslander algebra. Then the following are equivalent:
\begin{enumerate}
\item $\Lambda$ is a representation-directed algebra.
\item The associated category $\mathcal{A}$ of $A$ is a directed category.
\item $A$ is a quotient algebra of a finite-dimensional hereditary algebra.
\end{enumerate}
\end{proposition}

\begin{proof}
Let $M = \bigoplus _{i \in [n]} M_i$ be the direct sum of all indecomposable $\Lambda$-modules (up to isomorphism). Note that the associated category of $A = \End_{\Lambda} (M)$ has the following structure. Its objects are are indexed by $M_i$. By abuse of notation, we still denote these objects by $M_i$, $i \in[n]$. For two objects $M_i$ and $M_j$, $\mathcal{A} (M_i, M_j) = \Hom_{\Lambda} (M_i, M_j)$, which is an $(\End _{\Lambda} (M_j), \End_{\Lambda} (M_i))$-bimodule. Now it is straightforward to see that $\mathcal{A}$ is a directed category if and only if there is no sequences of nonisomorphic indecomposable $\Lambda$-modules $M_1, \ldots, M_t$ such that $\Hom _{\Lambda} (M_1, M_2) \neq 0$, $\ldots$, $\Hom _{\Lambda} (M_t, M_1) \neq 0$, i.e., every indecomposable $\Lambda$-module is not in a cycle. Therefore, (1) is equivalent to (2).

Clearly (3) implies (2). We finish the proof by showing (1) implies (3). We already know that $\mathcal{A}$ is a directed category. By Proposition 1.4 in Chapter IX \cite{ASS}, the endomorphism algebra of every directed module is one-dimensional, so $\End _{\Lambda} (M_i) \cong k$ for all $i \in [n]$. Therefore, $A$ is indeed a quotient algebra of a finite-dimensional hereditary algebra.
\end{proof}

Already given enough examples, we turn to study representations of $\mathcal{A}$. Recall a \textit{representation} $R$ of $\mathcal{A}$ is a $k$-linear covariant functor from $\mathcal{A}$ to $k$-vec. For $x \in \Ob \mathcal{A}$, the \textit{value} of $R$ on $x$ is defined as $R(x)$. The \textit{support} of $R$ is defined to be the set of objects $x$ such that $R(x) \neq 0$, and is denoted by $\supp (R)$. We say $R$ is \textit{generated} by its value on $x_1, \ldots, x_n$ if $R = \mathcal{A} \sum _{i \in [n]} R(x_i)$ and call $\{ x_i \} _{i \in [n]}$ a \textit{generating set} of $R$. If the set $\{ x_i \} _{i \in [n]}$ is contained in every generating set of $R$, it is called a \textit{minimal generating set} of $R$.

Note that the identity morphisms $1_x$, when $x$ ranges over all objects in $\mathcal{A}$, form a set of orthogonal idempotents in $A$, although they might not be primitive. Therefore, we have an $\mathcal{A}$-module decomposition $_{\mathcal{A}} \mathcal{A} \cong \bigoplus _{x \in \Ob \mathcal{A}} \mathcal{A} 1_x$, where $\mathcal{A} 1_x$ is the space of all morphisms starting from $x$. Let $E_x$ be a chosen set of orthogonal primitive idempotents in $\mathcal{A} (x, x)$ such that $\sum _{e \in E_x} e = 1_x$. Then $E = \sqcup _{x \in \Ob \mathcal{A}} E_x$ is a set of primitive orthogonal idempotents of $A$ with $\sum _{e \in E} e= 1$. Furthermore, the space constituted of all non-endomorphisms in $\mathcal{A}$ is a two-sided ideal of $A$. Therefore, the space constituted of all endomorphisms in $\mathcal{A}$ is a quotient algebra of $A$, and can be viewed as an $\mathcal{A}$-module. Also observe that for every $x \in \Ob \mathcal{A}$, a simple $\mathcal{A} (x, x)$-module can be lifted to a simple $\mathcal{A}$-module supported on $x$. These observations give us a description of indecomposable projective $\mathcal{A}$-modules and simple $\mathcal{A}$-modules.

\begin{proposition}
Let $\mathcal{A}$ be a connected finite directed category and $A$ be the associated algebra. Let $R$ be a representation of $\mathcal{A}$. Then:
\begin{enumerate}
\item Every indecomposable projective $\mathcal{A}$-module is isomorphic to $\mathcal{A} e$ with $e \in E_x$ for some $x \in \Ob \mathcal{A}$.
\item Every simple $\mathcal{A}$-module can be identified with a simple $\mathcal{A} (x, x)$-module for some $x \in \Ob \mathcal{A}$.
\item For every $x \in \Ob \mathcal{A}$, the value $R(x) = 1_x R \cong \Hom _{\mathcal{A}} (\mathcal{A}1_x, R)$.
\item The minimal generating set of $R$ exists, and is unique.
\item If $R$ is an indecomposable projective $\mathcal{A}$-module generated by $R(x)$, then $R(x)$ is a projective $\mathcal{A} (x, x)$-module.
\end{enumerate}
\end{proposition}

\begin{proof}
The first two statements are straightforward. The third statement follows from the equivalence between the category of representations of $\mathcal{A}$ and the category $A$-mod. It is clear that the minimal generating set of $R$ coincide with the support of the $\Top (R)$, where $\Top(R) = R / \rad R$, which clearly exists and is unique since the generating set of every simple module is a set containing a single object. If $R$ is indecomposable and projective, then $R \cong \mathcal{A} e$, where $e$ is a primitive idempotent in $\mathcal{A} (x, x)$. Therefore, $R(x) \cong 1_x \mathcal{A} e$ is a summand of $\mathcal{A} (x, x) = 1_x \mathcal{A} 1_x$ up to isomorphism.
\end{proof}

In the rest of this section we consider the behaviors of induction and restriction functors. Let $\mathcal{B}$ be a subcategory of $\mathcal{A}$, and let $V$ and $W$ be an $\mathcal{A}$-module and a $\mathcal{B}$-module respectively. The induction functor is $\uparrow _{\mathcal{B}} ^{\mathcal{A}} = \mathcal{A} \otimes _{\mathcal{B}} -$, sending $W$ to $\mathcal{A} \otimes _{\mathcal{B}} W$. Since the associated algebra $B$ of $\mathcal{B}$ is a subalgebra of $A$, this functor is well defined. On the other hand, the restriction functor $\downarrow _{\mathcal{B}} ^{\mathcal{A}}$ sends $V$ to $1_{\mathcal{B}} \cdot V$, which is a $\mathcal{B}$-module.

Suppose that $\mathcal{B}$ is a full subcategory of $\mathcal{A}$. We say $\mathcal{B}$ is an \textit{ideal} of $\mathcal{A}$ if whenever $x \in \Ob \mathcal{B}$, then every $y \in \Ob \mathcal{A}$ with $y \leqslant x$ is also contained in $\Ob \mathcal{B}$. Dually, we define \textit{co-ideals} of $\mathcal{A}$. It is not hard to see that if $\mathcal{B}$ is an ideal of $\mathcal{A}$, then the associated algebra $B$ is a right ideal of $A$. Dually, if $\mathcal{B}$ is a co-ideal of $\mathcal{A}$, then the associated algebra $B$ is a left ideal of $A$.

\begin{proposition}
Suppose that $\mathcal{B}$ is a (connected) full subcategory of $\mathcal{A}$. Let $V$ and $V'$ be $\mathcal{A}$-modules, and $W$ be a $\mathcal{B}$-module. We have:
\begin{enumerate}
\item $W \uparrow _{\mathcal{B}} ^{\mathcal{A}} \downarrow _{\mathcal{B}} ^{\mathcal{A}} \cong W$.
\item If $W$ is indecomposable, then $W \uparrow _{\mathcal{B}} ^{\mathcal{A}}$ is indecomposable.
\item If $\mathcal{B}$ is an ideal of $\mathcal{A}$, then $\downarrow _{\mathcal{B}} ^{\mathcal{A}}$ preserves left projective modules.
\item If $\mathcal{B}$ is a co-ideal of $\mathcal{A}$, then $\downarrow _{\mathcal{B}} ^{\mathcal{A}}$ preserves right projective modules.
\item If $\supp (V)$ is contained in $\Ob \mathcal{B}$, and $\mathcal{B}$ is a co-ideal of $\mathcal{A}$, then we have $\Ext _{\mathcal{A}} ^i (V, V') \cong \Ext _{\mathcal{B}} ^i (V \downarrow _{\mathcal{B}} ^{\mathcal{A}}, V' \downarrow _{\mathcal{B}} ^{\mathcal{A}})$ for $i \geqslant 0$.
\end{enumerate}
\end{proposition}

\begin{proof}
These results have been described in \cite{Xu1} in the context of finite EI categories, and the proofs are essentially the same. For details, please refer to that paper.

(1): By definition, we have
\begin{align*}
W \uparrow _{\mathcal{B}} ^{\mathcal{A}} \downarrow _{\mathcal{B}} ^{\mathcal{A}} & = 1_{\mathcal{B}} \cdot (\mathcal{A} \otimes _{\mathcal{B}} W) = 1_{\mathcal{B}} \mathcal{A} \otimes _{\mathcal{B}} 1_{\mathcal{B}} W\\
& = 1_{\mathcal{B}} \mathcal{A} 1_{\mathcal{B}} \otimes _{\mathcal{B}} W = \mathcal{B} \otimes _{\mathcal{B}} W \cong W
\end{align*}
since $\mathcal{B}$ is a full subcategory and $1_{\mathcal{B}} \mathcal{A} 1_{\mathcal{B}}$ can be identified with $\mathcal{B}$.

(2): Suppose $W \uparrow _{\mathcal{B}} ^{\mathcal{A}}$ is decomposable. Then we can write $W \uparrow _{\mathcal{B}} ^{\mathcal{A}} = M_1 \oplus M_2$, where both $M_1$ and $M_2$ are nonzero. But then
\begin{equation*}
W \cong W \uparrow _{\mathcal{B}} ^{\mathcal{A}} \downarrow _{\mathcal{B}} ^{\mathcal{A}} = M_1 \downarrow _{\mathcal{B}} ^{\mathcal{A}} \oplus M_2 \downarrow _{\mathcal{B}} ^{\mathcal{A}},
\end{equation*}
so either $M_1 \downarrow _{\mathcal{B}} ^{\mathcal{A}} = 0$ or $M_2 \downarrow _{\mathcal{B}} ^{\mathcal{A}} = 0$. Without loss of generality we assume $M_2 \downarrow _{\mathcal{B}} ^{\mathcal{A}} = 0$.

Let $G$ be the minimal generating set of $W$, so $G \subseteq \Ob \mathcal{B}$ and $W = \mathcal{B} \cdot \sum _{x \in G} W(x)$. Since
\begin{equation*}
W \uparrow _{\mathcal{B}} ^{\mathcal{A}} = \mathcal{A} \otimes _{\mathcal{B}} W = \mathcal{A} \otimes _{\mathcal{B}} (\mathcal{B} \cdot \sum _{x \in G} W(x)) = \mathcal{A} \cdot \sum _{x \in G} (1_{\mathcal{B}} \otimes _{\mathcal{B}} W(x)),
\end{equation*}
$W \uparrow _{\mathcal{B}} ^{\mathcal{A}}$ and hence $M_2$  are generated by its values on elements in $G$. But $M_2 \downarrow _{\mathcal{B}} ^{\mathcal{A}} = 0$ implies that the values of $M_2$ on all objects in $G \subseteq \Ob \mathcal{B}$ are all 0. Therefore, $M_2 = 0$. This contradiction tells us that $W \uparrow _{\mathcal{B}} ^{\mathcal{A}}$ is indecomposable.

(3): Let $P \cong \mathcal{A} e$ be a projective $\mathcal{A}$-module. Without loss of generality we can assume that $P$ is indecomposable, so $P \cong \mathcal{A}e$, where by the previous proposition $e$ is a primitive idempotent in $\mathcal{A} (x, x)$ for some $x \in \Ob \mathcal{A}$.

By definition, $P \downarrow _{\mathcal{B}} ^{\mathcal{A}} \cong 1_{\mathcal{B}} \mathcal{A} e$. Note that $1_{\mathcal{B}} \mathcal{A}$ constitutes of all morphisms in $\mathcal{A}$ ending at some object $y \in \Ob \mathcal{B}$. Since $\mathcal{B}$ is an ideal, by definition, there is no nonzero morphism in $\mathcal{A}$ staring from an object in $\Ob \mathcal{A} \setminus \Ob \mathcal{B}$ and ending at an object in $\Ob \mathcal{B}$. Therefore, $1_{\mathcal{B}} \mathcal{A} = 1_{\mathcal{B}} \mathcal{A} 1_{\mathcal{B}} = \mathcal{B}$, so $P \downarrow _{\mathcal{B}} ^{\mathcal{A}} \cong 1_{\mathcal{B}} \mathcal{A} e = \mathcal{B} e$. If $y \notin \Ob \mathcal{B}$, the last term in the above identity is 0. Otherwise, it is a nonzero projective $\mathcal{B}$-module.

(4): This is a dual statement of (3).

(5): First, since $\mathcal{B}$ is a co-ideal, every $\mathcal{B}$-module can be viewed as an $\mathcal{A}$-module, whose values on objects not contained in $\Ob \mathcal{B}$ are all zero. Conversely, given an $\mathcal{A}$-module whose values on objects not in $\Ob \mathcal{B}$ are all zero, it can be regarded as a $\mathcal{B}$-module.

By Eckmann-Shapiro Lemma, $\Ext _{\mathcal{B}} ^i (V \downarrow _{\mathcal{B}} ^{\mathcal{A}}, V' \downarrow _{\mathcal{B}} ^{\mathcal{A}}) \cong \Ext _{\mathcal{A}} ^i (V \downarrow  _{\mathcal{B}} ^{\mathcal{A}} \uparrow  _{\mathcal{B}} ^{\mathcal{A}}, V')$ for $i \geqslant 0$. By the above observation,
$V \downarrow  _{\mathcal{B}} ^{\mathcal{A}} \uparrow  _{\mathcal{B}} ^{\mathcal{A}} \cong V$, and the conclusion follows.
\end{proof}

An immediate result of this proposition is:

\begin{corollary}
If $\mathcal{A}$ is of finite representation type, so is every full subcategory.
\end{corollary}

\begin{proof}
Let $\mathcal{B}$ be a full subcategory of $\mathcal{A}$. If $\mathcal{B}$ has infinitely many non-isomorphic indecomposable representations, applying the induction functor, we get infinitely many non-isomorphic indecomposable representations by (2) of the above proposition. These induced indecomposable representations are non-isomorphic by (1) of the above proposition since restricted to $\mathcal{B}$ they are non-isomorphic. The conclusion follows.
\end{proof}

\section{Stratification properties}

In this section we study the stratification properties of directed categories. First we introduce some background knowledge on stratification theory. For more details, see \cite{CPS,DR,ES,Xi}.

Let $\Lambda$ be a finite-dimensional algebra and suppose that $_{\Lambda} \Lambda$ has $n$ indecomposable summands. Let $\preccurlyeq$ be a preorder on the set $[n] = \{ i \mid 1 \leqslant i \leqslant n\}$. For $i \in [n]$, we let $P_i$ be the corresponding indecomposable projective $\Lambda$-module, and let $S_i$ be its top. According to \cite{CPS}, $\Lambda$ is \textit{standardly stratified} with respect to $\preccurlyeq$ if there exist indecomposable modules $\Delta_i$, called \textit{standard modules}, such that the following conditions hold:
\begin{enumerate}
\item the number of composition factors $[\Delta_i, S_j] = 0$ unless $j \preccurlyeq i$ for $i, j \in [n]$;
\item there is an exact sequence $0 \rightarrow K_i \rightarrow P_i \rightarrow \Delta_i \rightarrow 0$ for every $i \in [n]$ such that $K_i$ has a filtration by standard modules $\Delta _j$ with $j \succ i$.
\end{enumerate}
If furthermore the endomorphism algebra of every standard module has dimension 1, then $\Lambda$ is called a \textit{quasi-hereditary algebra}.

Actually, the $i$-th standard module $\Delta_i$ can be defined as the largest quotient of $P_i$ all of whose composition factors $S_j$ satisfy $j \preccurlyeq i$. This works for arbitrary algebras $\Lambda$ even if it is not standardly stratified. The $i$-th \textit{proper standard module} $\overline{\Delta}_i$ is defined to be the largest quotient of $P_i$ all of whose composition factors $S_j$ satisfy $j \prec i$ except for a single copy of $S_i$, where $j \prec i$ means $j \preccurlyeq i$ but $i \npreceq j$. By considering indecomposable injective modules and largest submodules, we can define dually \textit{costandard modules} $\nabla_i$ and \textit{proper costandard modules} $\overline {\nabla}_i$. Let $\mathcal{F} _{\Lambda} (\Delta)$ be the full subcategory of $A$-mod such that each module in it has a filtration by standard modules. Similarly we define categories $\mathcal{F} _{\Lambda} (\overline {\Delta})$, $\mathcal{F} _{\Lambda} (\nabla)$, and $\mathcal{F} _{\Lambda} (\overline {\nabla})$.

It is clear that if $\Lambda$ is standardly stratified with respect to $\preccurlyeq$, then $_{\Lambda} \Lambda \in \mathcal{F} _{\Lambda} (\Delta)$. The converse of this statement is also true if the partial order associated to $\preccurlyeq$ is a linear order, as explained in 2.2.3 of \cite{CPS} and pp 12-13 of \cite{Webb2}. Since this condition holds in our context, we take the equivalent condition. That is, we say $\Lambda$ is \textit{standardly stratified} with respect to $\preccurlyeq$ if $_{\Lambda} \Lambda \in \mathcal{F} _{\Lambda} (\Delta)$. It is said to be \textit{properly stratified} if $_{\Lambda} \Lambda \in \mathcal{F} _{\Lambda} (\Delta) \cap \mathcal{F} _{\Lambda} (\overline {\Delta})$. The reader can see from the definition that quasi-hereditary algebras are properly stratified, and properly stratified algebras are standardly stratified.

Now let $\mathcal{A}$ be a connected finite directed category with the linear order $\leqslant$ on $\Ob \mathcal{A}$. This linear order $\leqslant$ induces a preorder $\preccurlyeq$ on the set of isomorphism classes of simple $\mathcal{A}$-modules as follows. Recall in Section 2 we have chosen a fixed set $E_x$ of primitive orthogonal idempotents with $\sum _{e \in E_x} e = 1_x$ for every object $x \in \Ob \mathcal{A}$, and defined $E$ to be the disjoint union of these sets. Therefore, for $e \in E_x$ and $e' \in E_y$, we let $e \preccurlyeq e'$ if $x \leqslant y$. The reader can check that $\preccurlyeq$ defined in this way is indeed a preorder, but in general is not a partial order. Moreover, if every object $x$ in $\mathcal{A}$ is \textit{primitive}, i.e., $1_x$ is a primitive idempotent, then $\preccurlyeq$ coincide with $\leqslant$. Therefore, $(E, \preccurlyeq)$ is a preordered set indexing all indecomposable summands of $_\mathcal{A} \mathcal{A}$. Note that $\mathcal{A}$ might have isomorphic indecomposable summands. This is allowed since if $P = \mathcal{A} e$ and $Q = \mathcal{A} f$ are isomorphic indecomposable projective $\mathcal{A}$-modules, then we can find an object $x$ and the corresponding set $E_x$ such that both $e$ and $f$ lie in $E_x$. Therefore, we have $e \preccurlyeq f$ and $f \preccurlyeq e$.

Results in the following proposition have been described in \cite{Li3} (see Section 4) and \cite{Li5}.

\begin{proposition}
Let $\mathcal{A}$ and $\preccurlyeq$ be as above. Then:
\begin{enumerate}
\item Every standard module is isomorphic to an indecomposable summand of $\mathcal{A} (x, x)$ for some $x \in \Ob \mathcal{A}$, where we identify $\bigoplus _{x \in \Ob \mathcal{A}} \mathcal{A} (x, x)$ with the quotient module $\mathcal{A} / J$ and $J$ is the two-sided ideal constituted of all non-endomorphisms in $\mathcal{A}$.
\item $\mathcal{A}$ is standardly stratified with respect to $\preccurlyeq$ if and only if $\mathcal{A} (x,y)$ is a projective $\mathcal{A} (y,y)$-module for all $x, y \in \Ob \mathcal{A}$.
\end{enumerate}
\end{proposition}

Note that every finite dimensional algebra $A$ can be regarded as a directed category $\mathcal{A}$ with one object $x$. The reader may want to know stratifications of this trivial category. Let us consider it in details to explain the above proposition. First, let us choose a set of primitive orthogonal idempotents $E = \{e_i\} _{i \in [n]}$ such that $1 = \sum _{i \in [n]} e_i$. Since there is only one object $A$, the linear order $\leqslant$ is trivial. Moreover, for $i, j \in [n]$, since $e_i$ and $e_j$ correspond to the same object $x$, we have $e_i \preccurlyeq e_j$ and $e_j \preccurlyeq e_i$ simultaneously. Therefore, the trivial linear order gives rise to the trivial preorder (not a partial order if $A$ is local) $\preccurlyeq$ on the chosen set of primitive orthogonal idempotents. Using the definition, we conclude that standard modules are precisely indecomposable projective modules, and $A$ is standardly stratified with respect to this trivial preorder.

In the rest of this section we assume that every object $x$ in $\mathcal{A}$ is primitive. By definition, the identity morphism $1_x$ is a primitive idempotent in the associated algebra $A$. Therefore, the endomorphism algebra $\mathcal{A} (x,x)$ is a finite-dimensional local algebra. Consequently, the associated algebra $A$ of $\mathcal{A}$ is a basic algebra since $\{ 1_x \} _{x \in \Ob \mathcal{A}}$ is a set of primitive orthogonal idempotents satisfying $\sum _{x \in \Ob \mathcal{A}} 1_x = 1$ and $A 1_x \cong A 1_y$ if and only if $x = y$. Moreover, the preorder $\preccurlyeq$ we defined before coincides with the given linear order $\leqslant$. Examples of these finite directed categories are described in \cite{Li4,Li5}. The following proposition asserts that these directed categories are characterized by their stratification properties.

\begin{proposition}
Let $A$ be a basic finite-dimensional algebra with $n$ isomorphism classes of simple modules. Let $\leqslant$ be a linear order on $[n]$. Then the following are equivalent:
\begin{enumerate}
\item Every standard module $\Delta_i$ has only composition factors isomorphic to $S_i$, $i \in [n]$.
\item Every proper standard modules $\overline{\Delta}_i$ is simple, i.e., isomorphic to $S_i$, $i \in [n]$.
\item The associated category $\mathcal{A}$ is a directed category with respect to $\leqslant$.
\end{enumerate}
\end{proposition}

Note that this is true even if $A$ is not standardly stratified.

\begin{proof}
Suppose that $\mathcal{A}$ is a directed category. Note that every object is primitive. By the previous proposition, every standard module $\Delta_i$ is supported on one object. Equivalently, $\Delta_i$ has only composition factors isomorphic to $S_i$. Clearly, $\overline {\Delta_i} \cong S_i$. Thus (3) implies (1) and (2). It is also clear that if $\Delta_i$ has composition factors not isomorphic to $S_i$, then the top of $\rad \Delta_i$ must have a simple summand not isomorphic to $S_i$. Consequently, $\overline {\Delta}_i$ has composition factors not isomorphic to $S_i$, and hence is not simple. Thus (2) implies (1).

Now we prove $(1)$ implies (3) by induction. Without loss of generality we assume that $n$ is the maximal element in $[n]$ with respect to $\leqslant$. The conclusion is trivially true for $n=1$. If $n >1$, take $e_n$ to be a primitive idempotent in $A$ such that $P_n = A e_n$ is a projective cover of $S_n$. Clearly, $P_n \cong \Delta_n$, so it has only composition factors isomorphic to $S_n$ by the given condition. It is straightforward to see that $A$ has the following description where $A_1 = (1-e_n) A (1-e_n)$ and $A_2 = e_n A e_n$.
\begin{equation*}
\xymatrix {\bullet \ar@(ld,lu)[]|{A_1} \ar[rr]^{e_n A (1-e_n)} & & \bullet \ar@(rd,ru)[]|{A_2}}
\end{equation*}
By induction hypothesis, the associated category $\mathcal{A}_1$ of $A_1$ is directed with respect to the linear order on $[n-1]$ inherited from $\leqslant$. Therefore, $\mathcal{A}$ is directed with respect to $\leqslant$.
\end{proof}

Since all proper standard modules are simple, $\mathcal{A}$ is actually properly stratified with respect to $\leqslant$. It is also straightforward to see that $\mathcal{A}$ is quasi-hereditary with respect to $\leqslant$ if and only if $A$ is a quotient algebra of a finite-dimensional hereditary algebra. Moreover, the reader can check that all costandard modules of $\mathcal{A}$ are precisely indecomposable injective modules.

For an arbitrary standardly stratified algebra $\Lambda$, it is well known that $\mathcal{F} _{\Lambda} (\Delta)$ is closed under direct summands, extensions, kernels of epimorphisms, but in general it is not closed under cokernels of monomorphisms. Actually, $\mathcal{F} _{\Lambda} (\Delta)$ has this property if and only if $\mathcal{F} _{\Lambda} (\Delta) = \mathcal{P} ^f (\Lambda)$, where $\mathcal{P} ^f {\Lambda}$ is the full subcategory of $\Lambda$-mod constituted of all objects with finite projective dimension. This simple observation gives a possible approach to answer the question of Platzeck and Reiten in \cite{PR}: under what conditions these two subcategories of $\Lambda$-mod coincide.

\begin{proposition}
Let $\Lambda$ be as above. Then the following are equivalent:
\begin{enumerate}
\item $\mathcal{F} _{\Lambda} (\Delta) = \mathcal{P}^f (\Lambda)$.
\item $\mathcal{F} _{\Lambda} (\Delta)$ is closed under the cokernels of monomorphisms.
\item The cokernel of every monomorphism $\iota: \Delta_i \rightarrow P$ is contained in $\mathcal{F} _{\Lambda} (\Delta)$, where $\Delta_i$ is a standard module and $P$ is an arbitrary projective module.
\end{enumerate}
\end{proposition}

\begin{proof}
The equivalence of (2) and (3) is the first statement of Theorem 0.3 in \cite{Li5}. Thus it is sufficient to show the equivalence of (1) and (2). If $\mathcal{F} _{\Lambda} (\Delta) = \mathcal{P}^f (\Lambda)$, then it is closed under cokernels of monomorphisms since $\mathcal{P}^f (\Lambda)$ has this property. Conversely, suppose that $\mathcal{F} _{\Lambda} (\Delta)$ has this property. Take an arbitrary $\Lambda$-module $M$ with $\pd _{\Lambda} M = n < \infty$ and consider a minimal projective resolution $P^{\bullet}$ of $M$. Clearly, $P^s = 0$ for $s > n$, and $\Omega ^n M \cong P^n \in \mathcal{F} _{\Lambda} (\Delta)$. By considering the exact sequence $0 \rightarrow \Omega^n M \rightarrow P^{n-1} \rightarrow \Omega ^{n-1} M \rightarrow 0$ we deduce that $\Omega ^{n-1} M \in \mathcal{F} _{\Lambda} (\Delta)$ since the first two terms lie in this category, and it is closed under cokernels of monomorphisms. Continuing this process we get $M \in \mathcal{F} _{\Lambda} (\Delta)$. Therefore, $\mathcal{P} ^f (\Lambda) \subseteq \mathcal{F} _{\Lambda} (\Delta)$. The other inclusion is clear.
\end{proof}

When $\mathcal{A}$ is a directed category, we have:

\begin{proposition}
If $\mathcal{A}$ is standardly stratified with respect to $\leqslant$, then:
\begin{enumerate}
\item $\mathcal{F} _{\mathcal{A}} (\Delta)$ is closed under cokernels of monomorphisms.
\item $\mathcal{F} _{\mathcal{A}} (\Delta) = \mathcal{P} ^f (\mathcal{A})$.
\item An $\mathcal{A}$-module $M$ has finite projective dimension if and only if for every $x \in \Ob \mathcal{A}$, $M(x)$ is a free $\mathcal{A} (x, x)$-module.
\end{enumerate}
\end{proposition}

\begin{proof}
The first statement is proved in Proposition 1.4 in \cite{Li5}, which implies the second one immediately. Now we prove (3).

Note that every standard module of $\mathcal{A}$ has the form of $\mathcal{A} (x, x)$ for some $x \in \Ob \mathcal{A}$. Therefore, if $M(x)$ is a free module for each $x \in \Ob \mathcal{A}$, it has a filtration by standard modules, so is contained in $\mathcal{F} _{\mathcal{A}} (\Delta) = \mathcal{P} ^f (\mathcal{A})$. Conversely, if $M \in \mathcal{F} _{\mathcal{A}} (\Delta) = \mathcal{P} ^f (\mathcal{A})$, then it has a filtration by standard modules, and from the description of standard modules we see $M(x) \cong \Delta_x ^s \cong \mathcal{A} (x, x)^s$, where $s = [M: \Delta_x]$
\end{proof}

Therefore, if $\mathcal{A}$ is standardly stratified with respect to $\leqslant$ and all objects are primitive, we have an explicit description for objects in $\mathcal{P} ^f (\mathcal{A})$. Unfortunately, for arbitrary finite directed categories, we cannot find such a description. Indeed, in the following example we show that for every linear order with respect to which $\mathcal{A}$ is standardly stratified, the category of modules with filtrations by standard modules is always a proper subcategory of $\mathcal{P} ^f (\mathcal{A})$.

\begin{example}
Let $A$ be the path algebra of the following quiver with relations $\delta^2 = \delta \gamma = \delta \epsilon = 0$ and $\gamma \alpha = \epsilon \beta$.
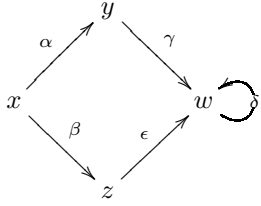
\begin{equation*}
\xymatrix{ & y \ar[dr] ^{\gamma} & \\ x \ar[ur] ^{\alpha} \ar[dr] ^{\beta} & & w \ar@(rd,ru)[]|{\delta} \\ & z \ar[ur] ^{\epsilon}}
\end{equation*}
Indecomposable projective modules are described as follows:
\begin{equation*}
P_x = \begin{matrix} & x & \\ y & & z \\ & w & \end{matrix}; \quad P_y = \begin{matrix} y \\ w \end{matrix}; \quad P_z = \begin{matrix} z \\ w \end{matrix}; \quad P_w = \begin{matrix} w \\ w \end{matrix}.
\end{equation*}
But for every linear with respect to which $A$ is standardly stratified, we can find an indecomposable module with finite projective dimension which does not have a filtration by standard modules. For example, if $y > x > z > w$, the standard modules are:
\begin{equation*}
\Delta_x = \begin{matrix} x \\ z \end{matrix}; \quad \Delta_y = \begin{matrix} y \\ w \end{matrix}; \quad \Delta_z = \begin{matrix} z \\ w \end{matrix}; \quad \Delta_w = \begin{matrix} w \\ w \end{matrix}.
\end{equation*}
Since $0 \rightarrow P_z \rightarrow P_x \rightarrow M = \begin{matrix} x \\ y \end{matrix} \rightarrow 0$ is exact, $\pd_A M = 1$. But $M$ has no filtration by standard modules.
\end{example}

The reader may want to know when a finite directed category $\mathcal{A}$ is quasi-hereditary with respect to the given linear order. The following proposition answer this question.

\begin{corollary}
Let $A$ be a basic finite-dimensional algebra. Then the following are equivalent:
\begin{enumerate}
\item $A$ is a quotient algebra of a finite-dimensional hereditary algebra.
\item $A$ is standardly stratified with respect to a linear order $\leqslant$, and all standard modules are simple.
\item $A$ is standardly stratified and $\mathcal{F} _A (\Delta) = A$-mod.
\end{enumerate}
\end{corollary}

\begin{proof}
$(1) \Rightarrow (2)$: If $A$ is a quotient algebra of a finite-dimensional hereditary algebra, then $\mathcal{A}$ is a finite directed category, and the endomorphism algebra of every object is isomorphic to $k$. Clearly, $\mathcal{A}$ is standardly stratified with respect to $\leqslant$ by (1) of Proposition 3.1. Moreover, all standard modules are simple by (2) of Proposition 3.1.

The equivalent of (2) and (3) is clear.

$(3) \Rightarrow (1)$: Since $\mathcal{F} _{A} (\Delta) = A$-mod, it is closed under cokernels of monomorphisms. Then (1) is true by Theorem 0.3 in \cite{Li5}.
\end{proof}

In \cite{Ringel} Ringel shows that for every finite-dimensional algebra $\Lambda$ standardly stratified with respect to a linear order, there is a \textit{characteristic tilting module} $T$, which is a generalized tilting module and is the \textit{Ext-injective} object in the category $\mathcal{F} _{\Lambda} (\Delta)$. That is, for every $M \in \mathcal{F} _{\Lambda} (\Delta)$, $\Ext _{\Lambda}^1 (M, T) = 0$. We end this section by describing explicitly the structure of this characteristic tilting module for some finite directed categories $\mathcal{A}$. Clearly, the indecomposable summands of $T$ can be indexed by $\Ob \mathcal{A}$. That is, for $x \in \Ob \mathcal{A}$ $T_x$ is the indecomposable summand of $T$ satisfying $[T_x : \Delta_x] = 1$ and $[T_x : \Delta_y] = 0$ for all $y \nleqslant x$, and $T = \bigoplus _{x \in \Ob \mathcal{A}} T_x$.

\begin{corollary}
Let $\mathcal{A}$ be a finite directed category such that every object is primitive and the endomorphism algebra is self-injective, and suppose that for all $x, y \in \Ob \mathcal{A}$, $\mathcal{A} (y, x)$ is a right free $\mathcal{A} (y,y)$-module. Then for every $x \in \Ob \mathcal{A}$, $T_x \cong \nabla_x \cong I_x$, where $I_x$ is the indecomposable injective $\mathcal{A}$-module corresponding to $x$.
\end{corollary}

\begin{proof}
We already know $\nabla_x \cong I_x$ (see the remark after Proposition 3.2). Also, since clearly $\Ext _{\mathcal{A}} ^i (M, I_x) = 0$ for $i \geqslant 1$ and $M \in \mathcal{F} _{\mathcal{A}} (\Delta)$, it is enough to show $I_x \in \mathcal{F} _{\mathcal{A}} (\Delta)$. Note that $I_x = D (Q_x)$, where $D$ is the functor $\Hom _{\mathcal{A}} (-, k)$, and $Q_x = 1_x \mathcal{A}$ is the space of morphisms ending at $x$. As a right $\mathcal{A}$-module (or equivalently, a left $\mathcal{A} ^{\text{op}}$-module), the value $Q_x (y)$ of $Q_x$ on an arbitrary object $y$ is $\mathcal{A} ^{op} (x, y) = \mathcal{A} (y, x)$, which is a right free $\mathcal{A} (y,y)$-module. Therefore, the value $I_x (y)$ is a left free $D (\mathcal{A} (y, y))$-module. But $\mathcal{A} (y, y)$ is local and self-injective, so it is a Frobenius algebra. Therefore, $I_x (y)$ is actually a left free $\mathcal{A} (y, y)$-module. The conclusion then follows from Proposition 3.4.
\end{proof}

This result is trivially true if $\mathcal{A}$ is quasi-hereditary with respect to the given linear order $\leqslant$. Indeed, in this case $\mathcal{F} _{\mathcal{A}} (\Delta)$ contains all $\mathcal{A}$-modules.

The condition that for all $x, y \in \Ob \mathcal{A}$, $\mathcal{A} (y, x)$ is a right free $\mathcal{A} (y,y)$-module is equivalent to saying that the opposite category $\mathcal{A} ^{\text{op}}$ is standardly stratified with respect to the opposite linear order $\leqslant ^{\text{op}}$. It is also equivalent to saying that the right projective dimension of $\mathcal{A} ^{\text{op}} (x, x)$ as an $\mathcal{A} ^{\text{op}}$-module is finite. This condition cannot be dropped, as explained by the following example. This example also tells us that the associated category of the Ringel dual $\End _{\mathcal{A}} (T) ^{\text{op}}$ might not be a directed category.

\begin{example}
Let $A$ be the following path algebra with relations $\alpha \delta = \delta^2 = 0$.
\begin{equation*}
\xymatrix {x \ar@(ld,lu)[]|{\delta} \ar[r]^{\alpha} & y}
\end{equation*}
It is easy to check that $\mathcal{A}$ is standardly stratified with respect to the order $x < y$. Projective modules, standard modules, and injective modules are as follows:
\begin{equation*}
P_x = \begin{matrix} & x & \\ x & & y \end{matrix} \quad P_y = y \quad \Delta_x = \begin{matrix} x \\ x \end{matrix} \quad \Delta_y = y \quad I_x = \begin{matrix} x \\ x \end{matrix} \quad I_y = \begin{matrix} x \\ y \end{matrix}
\end{equation*}
Clearly, $I_y$ is not contained in $\mathcal{F} _{\mathcal{A}} (\Delta)$. Actually, the characteristic tilting module is $T \cong P_x \oplus I_x$.

Now let us consider $\Gamma = \End _{\mathcal{A}} (T) ^{\text{op}}$. It is isomorphic to the path algebra of the following quiver with relation $\beta \alpha \beta = 0$, whose associated category is not directed.
\begin{equation*}
\xymatrix{1 \ar@/^/[rr]^{\alpha} & & 2 \ar@/^/[ll]^{\beta}.}
\end{equation*}
With respect to the order $1 < 2$ the algebra $\Gamma$ is standardly stratified. Its indecomposable projective modules, injective modules, and standard modules are listed below:
\begin{equation*}
P_1 = \begin{matrix} 1 \\ 2 \\ 1 \\ 2 \end{matrix} \quad P_2 = \begin{matrix} 2 \\ 1 \\ 2 \end{matrix} \quad \Delta_1 = 1 \quad \Delta_2 = \begin{matrix} 2 \\ 1 \\2 \end{matrix} \quad I_1 = \begin{matrix} 1 \\ 2 \\ 1 \end{matrix} \quad I_2 = \begin{matrix} 1 \\ 2 \\ 1 \\2 \end{matrix}
\end{equation*}
The characteristic tilting $\Gamma$-module $T' \cong P_1 \oplus \Delta_1$. The opposite algebra of $\End _{\Gamma} (T')$ is isomorphic to $A$, as claimed by Ringel's duality.
\end{example}

\section{Generalized APR tilting modules}

Our main goal in this section is to prove the existence of generalized APR tilting modules for triangular matrix algebras. Let $A$ be a basic finite-dimensional algebra. In \cite{APR} it is shown that if $A$ has a simple projective module $S$, then $T = Q \oplus \tau ^{-1} S$ is a tilting module, called the \textit{APR tilting module}, where $Q$ is the direct sum of all indecomposable summands of $_AA$ not isomorphic to $S$, and $\tau$ is the Auslander-Reiten translation. In particular, if $A$ is a hereditary algebra, then the functor $\Hom _A (T, -)$ is precisely the BGP reflection functor (\cite{ASS,APR}).

Let $e$ be a primitive idempotent in $A$ with $Ae \cong S$ and $\epsilon = 1 - e$. A simple observation tells us that $A = \begin{bmatrix} \epsilon A \epsilon & 0 \\ eA\epsilon & k \end{bmatrix}$ is a triangular matrix algebra with the following structure (called \textit{one-point trivial extension}):
\begin{equation*}
\xymatrix {\bullet \ar@(ld,lu)[]|{\epsilon A \epsilon} \ar[rr]^{eA \epsilon} & & \bullet}
\end{equation*}
Therefore, we may ask whether a generalized APR tilting module exists if $A$ has a projective module $P_S$ all of whose composition factors are isomorphic to the simple module $S$. In other words, there is a primitive idempotent $e$ in $A$ such that $A \cong \begin{bmatrix} \epsilon A \epsilon & 0 \\ eA\epsilon & eAe \end{bmatrix}$ and $P_S = Ae$. Structure of $A$ can be pictured as below:
\begin{equation*}
\xymatrix {\bullet \ar@(ld,lu)[]|{\epsilon A \epsilon} \ar[rr]^{eA \epsilon} & & \bullet \ar@(rd,ru)[]|{eAe}}.
\end{equation*}

We introduce some notations here. Let $\mathcal{A}$ be the associated category of $A$. Then $\{ 1_x \} _{x \in \Ob \mathcal{A}}$ is a set of primitive orthogonal idempotents in $A$. Let $z$ be the object on which $P_S = Ae$ is supported. That is, $\mathcal{A} 1_z \cong Ae \cong P_S$. Note that $\mathcal{A}$ need not be a directed category. However, we always have $\mathcal{A} (z, x) = 0$ for $z \neq x \in \Ob \mathcal{A}$. Let $\mathcal{B}$ be the full subcategory of $\mathcal{A}$ constituted of all objects $x$ different from $z$. Then the associated algebra of $B$ is exactly $\epsilon A \epsilon$, and $\mathcal{A}$ has the following description.
\begin{equation*}
\xymatrix {\mathcal{B} \ar[rr]^{eA \epsilon} & & z \ar@(rd,ru)[]|{eAe}}
\end{equation*}
Let $A^{\text{o}}$ and $\mathcal{A} ^{\text{o}}$ be the opposite algebra of $A$ and the opposite category of $\mathcal{A}$ respectively. Let $Q = A \epsilon$, which is the direct sum of all other indecomposable summands of $_AA$ not isomorphic to $P_S$. Define $T = Q \oplus \tau ^{-1} P_S$. Now the problem is to check under what conditions $T$ is a tilting module. Since $T$ has $n$ pairwise nonisomorphic indecomposable summands, it suffices to check $\pd _A T \leqslant 1$ and $\Ext _A^1 (T, T) = 0$.

The following lemmas are crucial for the main result of this section.

\begin{lemma}
The projective dimension $\pd_A T \leqslant 1$ if and only if $\Hom _{A^{\text{o}}} (DP_S, A^{\text{o}}) = 0$, where $D = \Hom_k (-, k)$.
\end{lemma}

\begin{proof}
Since by our construction $T$ is the direct sum of $\tau ^{-1} P_S$ and some projective modules, $\pd_A T \leqslant 1$ if and only if $\pd_A \tau ^{-1} P_S \leqslant 1$. Take a projective presentation $P^1 \rightarrow P^0 \rightarrow DP_S \rightarrow 0$, where all modules are left $A^{\text{o}}$-modules (or right $A$-modules). This presentation gives rise to the following exact sequence
\begin{equation*}
(\ast): 0 \rightarrow \Hom _{A^{\text{o}}} (DP_S, A^{\text{o}}) \rightarrow \Hom_{A^{\text{o}}} (P^0, A^{\text{o}}) \rightarrow \Hom_{A^{\text{o}}} (P^1, A^{\text{o}}) \rightarrow \tau^{-1} P_S \rightarrow 0.
\end{equation*}
Note that the second term and the third term are projective $A$-modules. Therefore, it is easy to see that $\Hom _{A^{\text{o}}} (DP_S, A^{\text{o}}) = 0$ implies $\pd_A T \leqslant 1$.

Conversely, if $\pd_A T \leqslant 1$, from the above exact sequence we conclude that $\Hom _{A^{\text{o}}} (DP_S, A^{\text{o}})$ is a projective module. Note that $DP_S = D(eAe)$ is a module only supported on $z$, the first object of $\mathcal{A} ^{\text{o}}$ having the following structure:
\begin{equation*}
\xymatrix {\mathcal{A} ^{\text{o}}: & z \ar@(ld,lu)[]| {(eAe) ^{\text{o}}} \ar[rr] & & \mathcal{B} ^{\text{o}}},
\end{equation*}
$A^{\text{o}} \cong A^{\text{o}}e \oplus B^{\text{o}}$ as left modules, and $\Hom _{A^{\text{o}}} (DP_S, B^{\text{o}}) = 0$, so $\Hom _{A^{\text{o}}} (DP_S, A^{\text{o}}) \cong \Hom _{A^{\text{o}}} (DP_S, A^{\text{o}} e)$.

From our construction, $DP_S \cong (eAe)^{\text{o}}$ has the following short exact sequence:
\begin{equation*}
\xymatrix{0 \ar[r] & M \ar[r] & A^{\text{o}}e \ar[r] & DP_S \ar[r] & 0}
\end{equation*}
with $M \neq 0$. Applying the functor $\Hom _{A^{\text{o}}} (-, A^{\text{o}} e)$ we get:
\begin{align*}
& 0 \rightarrow \Hom _{A^{\text{o}}} (DP_S, A^{\text{o}} e) \rightarrow \Hom _{A^{\text{o}}} (A^{\text{o}}e, A^{\text{o}} e)\\
& \rightarrow \Hom _{A^{\text{o}}} (M, A^{\text{o}} e) \rightarrow \Ext^1 _{A^{\text{o}}} (DP_S, A^{\text{o}}e) \rightarrow 0.
\end{align*}
But $DP_S$ is actually an injective $A^{\text{o}}$-module because $eAe$ is self-injective, so the extension group $\Ext^1 _{A^{\text{o}}} (DP_S, A^{\text{o}}e) = 0$. Clearly, $\Hom _{A^{\text{o}}} (M, A^{\text{o}} e) \neq 0$. Therefore, $\Hom _{A^{\text{o}}} (DP_S, A^{\text{o}}) \cong \Hom _{A^{\text{o}}} (DP_S, A^{\text{o}} e)$ is a proper submodule of the indecomposable projective $A$-module $\Hom _{A^{\text{o}}} (A^{\text{o}}e, A^{\text{o}} e) \cong eAe = Ae$. From the structure of $A$ we conclude that the only possibility for $\Hom _{A^{\text{o}}} (DP_S, A^{\text{o}})$ to be projective is that it is actually 0. This finishes the proof.
\end{proof}

\begin{lemma}
The condition $\Hom _{A^{\text{o}}} (DP_S, A^{\text{o}}) = 0$ holds if there is some $x \in \Ob \mathcal{A}$ such that $\mathcal{A} (x, z)$ as a left $\mathcal{A} (z, z)$-module has a free summand.
\end{lemma}

\begin{proof}
Note that $DP_S$ as a left $\mathcal{A} ^{\text{o}}$-module is only supported on $z$, which is a minimal object in $\Ob \mathcal{A} ^{\text{o}}$. Therefore, for all $z \neq x \in \Ob \mathcal{A} ^{o}$,
\begin{equation*}
\Hom _{A^{\text{o}}} (DP_S, A^{\text{o}} 1_x) = \Hom _{\mathcal{A} ^{\text{o}}} (DP_S, \mathcal{A} ^{\text{o}} 1_x) = 0
\end{equation*}
since $\mathcal{A} ^{\text{o}} 1_x$ is supported on objects different from $z$. Therefore, $\Hom _{A^{\text{o}}} (DP_S, A^{\text{o}}) = 0$ if and only if $\Hom _{\mathcal{A} ^{\text{o}}} (DP_S, \mathcal{A} ^{\text{o}} 1_z) = 0$.

Let $z \neq x \in \Ob \mathcal{A} ^{\text{o}}$ be an object such that $\mathcal{A} (x, z)$ has a free summand as a left $\mathcal{A} (z, z)$-module. Take $\varphi \in \Hom _{\mathcal{A} ^{\text{o}}} (DP_S, \mathcal{A} ^{\text{o}} 1_z)$. The homomorphism $\varphi$ gives the following diagram by considering the values on $x$ and $z$:
\begin{equation*}
\xymatrix{ DP_S(z) = D(eAe) \ar[rr] \ar[d]^{\varphi_z} & & DP_S(x) = 0 \ar[d]^{\varphi_x} \\
\mathcal{A} ^{\text{o}} 1_z (z) = (eAe) ^{\text{o}} \ar[rr] ^{\rho} & & \mathcal{A} ^{\text{o}} 1_z (x) = \mathcal{A} (x, z).}
\end{equation*}

Note the given condition tells us that $\mathcal{A} ^{\text{o}} (z, x) = \mathcal{A} (x, z)$ has a free summand as a right $\mathcal{A} ^{\text{o}} (z, z)$-module, where the right action of $\mathcal{A} ^{\text{o}} (z, z)$ on $\mathcal{A} ^{\text{o}} (z, x) = \mathcal{A} (x, z)$ is defined as follow: for $\delta \in \mathcal{A} ^{\text{o}} (z, z)$ and $\alpha \in \mathcal{A} ^{\text{o}} (z, x)$, $\alpha \ast \delta = \delta \alpha$. Write $\mathcal{A} ^{\text{o}} (z, x) = \bigoplus _{i \in [s]} M_i$ as right $\mathcal{A} ^{\text{o}} (z, z)$-modules and without loss of generality suppose that $M_1$ is a right free summand. Therefore, the map $\rho$ is determined by $s$ morphisms $\alpha_1, \ldots, \alpha_s$ in $\mathcal{A} ^{\text{o}} (z, z)$ such that for every $\delta \in \mathcal{A} ^{\text{o}} (z, z)$,
\begin{equation*}
\rho (\delta) = \sum _{i \in [s]} \alpha_i \ast \delta = \sum _{i \in [s]} \delta \alpha_i
\end{equation*}
with $\delta \alpha_i \in M_i \subseteq \mathcal{A} ^{\text{o}} (z, x)$. In particular, since $M_1$ is a right free $\mathcal{A} ^{\text{o}} (z, z)$-module, $\alpha_1$ induces a bijection between $\mathcal{A} ^{\text{o}} (z, z)$ and $M_1$. Therefore, for $0 \neq \delta \in \mathcal{A} ^{\text{o}} (z, z)$, $\delta \alpha_1 \neq 0$, so $\rho (\delta) \neq 0$, and hence $\rho$ is injective. Consequently, from the above diagram we conclude $\varphi_z = 0$, so $\varphi = 0$. This finishes the proof.
\end{proof}

Now we are ready to prove the main result.

\begin{theorem}
Notation as before. Suppose that $\mathcal{A} (z, z) = eAe$ is a self-injective algebra, and there is some $z \neq x \in \Ob \mathcal{A}$ such that $\mathcal{A} (x, z)$ has a free summand as a left $\mathcal{A} (z,z)$-module. Then $T$ is a tilting module.
\end{theorem}

\begin{proof}
Under the given assumptions, we have shown that $\pd_A T \leqslant 1$, so it suffices to show $\Ext _A^1 (T, T) = 0$. We have:
\begin{align*}
\Ext_A^1 (T, T) & = \Ext_A^1 (Q \oplus \tau ^{-1} P_S, Q \oplus \tau ^{-1} P_S) = \Ext_A^1 (\tau ^{-1} P_S, Q \oplus \tau ^{-1} P_S) \\
& \cong D \overline{\Hom} _A (Q \oplus \tau ^{-1} P_S, P_S),
\end{align*}
where the last isomorphism follows from the Auslander-Reiten formula (see \cite{ARS}), and $\overline{\Hom} _A (Q \oplus \tau^{-1} P_S, P_S)$ is the quotient space of $\Hom _A (Q \oplus \tau^{-1} P_S, P_S)$ modulo all homomorphisms factoring through injective $A$-modules. Therefore, it is sufficient to show $\Hom _A (Q \oplus \tau^{-1} P_S, P_S) = 0$. Observing the structure of $\mathcal{A}$ we conclude $\Hom _A (Q, P_S) = \Hom _A (A \epsilon, Ae) \cong \epsilon A e = 0$. Therefore, it is enough to show $\Hom _A (\tau^{-1} P_S, P_S) = 0$.

In the proof of Lemma 4.1 we have actually constructed a projective resolution for $\tau ^{-1} P_S$:
\begin{equation*}
(\ast): 0 \rightarrow \Hom_{A^{\text{o}}} (P^0, A^{\text{o}}) \rightarrow \Hom_{A^{\text{o}}} (P^1, A^{\text{o}}) \rightarrow \tau^{-1} P_S \rightarrow 0.
\end{equation*}
Let $Q' = \Hom_{A^{\text{o}}} (P^1, A^{\text{o}})$. Thus the conclusion will follow if $\Hom_A (Q', P_S) = 0$.

Since $P_S$ is only supported on $z$, so is $DP_S$. Moreover, $DP_S (z) = D (eAe) \cong (eAe) ^{\text{o}}$ since $eAe$ is local and self-injective. Therefore, $P^0 \cong eA = 1_z \mathcal{A}$ and $P^0 (z) \cong DP_S (z)$, so the first syzygy $\Omega DP_S$ and hence $P^1$ are supported on objects different from $z$. We can write $P^1 \cong \bigoplus _{z \neq x \in \Ob \mathcal{A}} (1_x A) ^{n_x}$, $n_x \geqslant 0$. Consequently,
\begin{equation*}
Q' = \Hom_{A^{\text{o}}} (P^1, A^{\text{o}}) \cong \Hom_{A ^{\text{o}}} \Big{(} \bigoplus _{z \neq x \in \Ob \mathcal{A}} (1_x A) ^{n_x}, A^{\text{o}} \Big{)} \cong \bigoplus _{z \neq x \in \Ob \mathcal{A}} (A 1_x) ^{n_x}
\end{equation*}
is a direct sum of summands of $Q$. But we have shown $\Hom _A (Q, P_S) = 0$, so  $\Hom _A (Q', P_S) = 0$. This finishes the proof.
\end{proof}

In the proof of this theorem we have shown $\Hom _A (\tau^{-1} P_S, P_S) = 0$. Therefore, $\Hom _A (\tau^{-1} P_S, S) = 0$. Indeed, since by the assumption $P_S$ has only composition factors isomorphism to $S$, in particular its socle contains a simple summand isomorphic to $S$ and there is an inclusion $S \rightarrow P_S$. If $\Hom _A (\tau^{-1} P_S, S) \neq 0$, then $\Hom _A (\tau^{-1} P_S, P_S) \neq 0$ either. This is impossible. Consequently, for any $\mathcal{A}$-module $M$ which is only supported on $z$, or equivalently, which only has composition factors isomorphic to $S$, we have $\Hom_A (\tau ^{-1} P_S, M) = 0$.

The generalized APR tilting module $T$ induces a torsion theory $(\mathcal{T}, \mathcal{F})$, where $\mathcal{T}$ constitutes of all quotient modules of $T^s$ for some $s \geqslant 0$, and $\mathcal{F}$ is formed by all $A$-modules $M$ such that $\Hom_A (T, M) = 0$. We have:

\begin{corollary}
The category $\mathcal{F}$ constitutes of all $A$-modules $M$ all of whose composition factors are isomorphic to $S$.
\end{corollary}

\begin{proof}
Suppose that $M$ only has composition factors isomorphic to $S$. Clearly, $\Hom_A (Q, M) = 0$. But we also have $\Hom_A (\tau ^{-1} P_S, M) = 0$, so $\Hom_A (T, M) = 0$, and $M$ is contained in $\mathcal{F}$. Conversely, for every $A$-module $X$ having a composition factor $T$ not isomorphic to $S$, there exists a summand $Q_T$ of $Q$ such that $\Hom_A (Q_T, X) \neq 0$, so $X$ is not contained in $\mathcal{F}$.
\end{proof}

The torsion theory $(\mathcal{T}, \mathcal{F})$ induced by the generalized APR tilting module $T$ in general is neither \textit{separating} nor \textit{splitting} (see \cite{ASS} for definitions). Indeed, By Theorem 5.6 in Page 230 \cite{ASS}, it is splitting if and only if the injective dimension $\id_A M \leqslant 1$ for every $M \in \mathcal{F}$. If $\mathcal{A} (z, z) = eAe$ is not isomorphic to $k$, then the simple module $S$ (which is in $\mathcal{F}$) has infinite injective dimension.

We end this section with two examples.

\begin{example}
Let $\mathcal{A}$ be pictured below with relations $\beta \rho = \delta \beta$, $\delta^2 = \rho^2 = \delta \alpha = 0$.
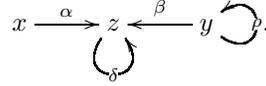
\begin{equation*}
\xymatrix {x \ar[r] ^{\alpha} & z \ar@(dl,dr)[]|{\delta} & y \ar@(rd,ru)[]|{\rho} \ar[l] _{\beta}}.
\end{equation*}
Left projective modules $P_x, P_y, P_z$ are pictured as below:
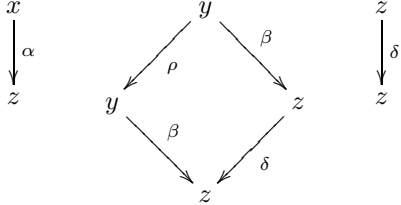
\begin{equation*}
\xymatrix {x \ar[d]^{\alpha} \\ z} \qquad
\xymatrix{ & y \ar[dl] ^{\rho} \ar[dr] ^{\beta} & \\ y \ar[dr] ^{\beta} & & z \ar[dl] ^{\delta} \\ & z & } \qquad
\xymatrix {z \ar[d]^{\delta} \\ z}
\end{equation*}
Right projective modules $Q_x, Q_y, Q_z$ are pictured as below:
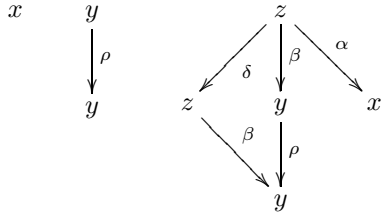
\begin{equation*}
x \qquad
\xymatrix {y \ar[d]^{\rho} \\ y} \qquad
\xymatrix{ & z \ar[dl] ^{\delta} \ar[d] ^{\beta} \ar[dr] ^{\alpha} & \\ z \ar[dr] ^{\beta} & y \ar[d] ^{\rho} & x \\ & y & }
\end{equation*}
By the above theorem, $P_x \oplus P_y \oplus \tau ^{-1} P_z$ is a generalized APR tilting module.

Let $A$ be the associated algebra of $\mathcal{A}$. As a right $A$-module, $DP_z$ has the following projective presentation:
\begin{equation*}
\xymatrix{0 \ar[r] & Q_x \oplus Q_y \ar[r] & Q_z \ar[r] & DP_z \ar[r] & 0}.
\end{equation*}
Applying the functor $\Hom _{A^{\text{o}}} (-, A^{\text{o}})$, we get a projective resolution of $\tau ^{-1} P_z$:
\begin{equation*}
\xymatrix{0 \ar[r] & P_z \ar[r] & P_x \oplus P_y \ar[r] & \tau^{-1} P_z \ar[r] & 0}.
\end{equation*}
Thus $\pd_A \tau^{-1} P_z = 1$. It is not hard to see that $\tau^{-1} P_z$ has the following structure:
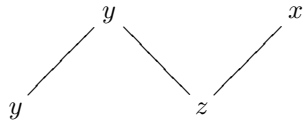
\begin{equation*}
\xymatrix{ & y \ar@{-}[dl] \ar@{-}[dr] & & x \ar@{-}[dl] \\ y & & z & }
\end{equation*}
Clearly, $\Hom_A (\tau ^{-1} P_z, P_z) = 0$.
\end{example}

If $P_S$ is simple, then the almost split sequence starting at $P_S$ is precisely the projective resolution of $\tau^{-1} P_S$ (see \cite{ASS}). This is not true for generalized APR tilting modules, as shown by the following example:

\begin{example}
Let $A$ be the path algebra of the following quiver with relations $\delta \alpha = \alpha \rho$ and $\delta^2 = \rho^2 = 0$.
\begin{equation*}
\xymatrix {x \ar@(ld,lu)[]|{\rho} \ar[r] ^{\alpha} & y \ar@(rd,ru)[]|{\delta}}.
\end{equation*}
Then $P_x \oplus \tau ^{-1} P_y$ is a generalized APR tilting module. By computation, $\tau^{-1} P_y$ coincides with the injective module $I_x$. But the almost split sequence ending at $\tau^{-1} P_y$ is:
\begin{equation*}
\xymatrix{0 \ar[r] & P_y \ar[r] & M \ar[r] & \tau^{-1} P_y \cong I_x \ar[r] & 0},
\end{equation*}
where $M$ has the following structure:
\begin{equation*}
\xymatrix{ & x \ar@{-}[dl] \ar@{-}[dr] & & y \ar@{-}[dl] \\ x & & y & }.
\end{equation*}
Therefore, the almost split sequence is not the projective resolution of $\tau^{-1} P_x$:
\begin{equation*}
\xymatrix{0 \ar[r] & P_y \ar[r] & P_x \ar[r] & \tau^{-1} P_y \cong I_x \ar[r] & 0}.
\end{equation*}
\end{example}

\end{document}